\documentclass[11pt]{amsart}
 
\usepackage{latexsym,amssymb,amsxtra,verbatim, graphics,rotating,multirow}%

\newcommand{\ds}{\oplus}
\newcommand{\sR}{\mathcal{R}}
\newcommand{\Tr}{\mathrm{Tr}} %

\newcommand{\so}{\mathrm{SO}}
\newcommand{\BBF}{\mathbb{F}}

\newcommand{\orth}{\mathrm{GO}}
\newcommand{\co}{\mathrm{CO}}
\newcommand{\gl}{\mathrm{GL}}
\renewcommand{\sl}{\mathrm{SL}}
\newcommand{\spl}{\mathrm{Sp}}

\newcommand{\su}{\mathrm{SU}}
\newcommand{\gu}{\mathrm{GU}}

\newcommand{\la}{\lambda}
\newcommand{\F}{\mathbb{F}}
\newcommand{\prim}{\xi}
\newcommand{\primtwo}{\zeta}

\newcommand{\normeq}{\unlhd} %
\DeclareMathOperator{\CC}{{C}}

\newtheorem{thm}{Theorem}[section]
\newtheorem{lem}[thm]{Lemma}
\newtheorem{prop}[thm]{Proposition}
\newtheorem{cor}[thm]{Corollary}
\newtheorem{df}[thm]{Definition}

\DeclareMathOperator{\diag}{diag}
\DeclareMathOperator{\spin}{sp}
\DeclareMathOperator{\refl}{refl}
\DeclareMathOperator{\rank}{rank}
\DeclareMathOperator{\antidiag}{antidiag}

\title{Constructive homomorphisms for classical groups}

\author{Scott H.\ Murray}
\address{Department of Mathematics and Statistics, Building 11,
University of Canberra, ACT, 2601, Australia} 
\email{\tt murray@maths.usyd.edu.au}
\author{Colva M.\ Roney-Dougal}
\address{School of Mathematics and Statistics, University of St Andrews, Fife KY16 9SS, UK.}
\email{\tt colva@mcs.st-and.ac.uk}
\date{\today}

\begin{document}

\begin{abstract}
Let $\Omega\leq\gl(V)$ be a quasisimple classical group 
in its natural representation over a finite vector space $V$, and let $\Delta=\mathrm{N}_{\gl(V)}(\Omega)$.
We construct the projection from $\Delta$ to $\Delta/\Omega$
and provide fast, polynomial-time algorithms for computing the image
of an element.
Given a discrete logarithm oracle, we also represent 
$\Delta/\Omega$ as a group with at most 3 generators and 6 relations. We then compute canonical representatives for the cosets 
of~$\Omega$. 
 A key ingredient of our algorithms is a new, asymptotically fast method for constructing isometries between spaces with forms.  Our results are useful for the matrix group recognition project,  can be used to solve element conjugacy problems, and can improve algorithms to construct maximal subgroups.
\end{abstract}

\thanks{We would like to thank the anonymous referees, whose careful reading and helpful suggestions have significantly improved this paper. We also thank the
Magma project at the University of Sydney, where some of the work was carried
out.
The second author would like to acknowledge the support of the Nuffield Foundation, and of EPSRC grant EP/C523229/1. }

\maketitle

\section{Introduction}\label{sec:intro}
In this paper, we provide a variety of algorithms for classical groups.
Fix a prime $p$ and a power $q$ of $p$, and let $u=2$ for unitary groups and $1$ otherwise.
We consider groups $H\le\gl_d(q^u)$ such that $\Omega\le H\le \Delta$,
where $\Omega$ is a \emph{quasisimple classical group} and $\Delta=\mathrm{N}_{\gl_d(q^u)}(\Omega)$ is the corresponding \emph{conformal group} \cite[Section 2.1]{KL90}. 
Most of our algorithms are randomised Las Vegas in the sense of \cite{Babai97}.
We often need Las Vegas algorithms whose output is independent of the random choices made.
In this case we call the output \emph{canonical}.

The matrix group recognition project \cite{Leedham-Green01} seeks to compute efficiently  composition series for 
matrix groups over finite fields. 
By finding a geometry preserved by the group, in the sense of Aschbacher's theorem \cite{Asch}, 
a normal subgroup and its quotient can often be computed. 
This decomposition terminates on reaching %
 groups that are almost simple,
modulo their subgroup of scalar matrices.
These groups are either classical groups  
in their natural representation (Aschbacher's class~8) or other almost simple groups (class~9).
This paper provides algorithms for dealing with a group known to be in class~8. 
Algorithms to constructively recognise the quasisimple classical groups %
in their natural representation are known \cite{BrooksOmega, BrooksClassical}. 
This paper presents efficient, practical reduction algorithms for the other class~8 groups.

Another motivation is constructing efficient algorithms  
for element conjugacy in classical groups $H$, when the dimension $d$ is large.
The fundamental problem is to determine if two elements are conjugate and, if so, provide a conjugating element.
For the sake of memory efficiency, it makes sense to conjugate
a single element to a canonical representative of its conjugacy class.
Given a solution to this conjugacy problem for $\Delta$ 
\cite{HallerMurray07Pre,Britnell06B},
we can construct an algorithm to solve the element conjugacy algorithm in a group $H$ between $\Omega$ and $\Delta$,
provided that %
we have \emph{canonical} coset representatives for $H/\Omega$. 
This, along with applications to the construction of maximal subgroups,  are the primary motivations for the requirement that our algorithms give canonical 
solutions. 
See Section~\ref{sect:app} for more details.

We give our timings in terms of elementary finite field operations: addition, negation, multiplication, and inversion.
The number of field operations required by our algorithms is polynomial in $d$ and $\log q$, except for some algorithms which require calls to a discrete logarithm oracle.   We specify when this is the case, and count the number of calls to the oracle. %

We consider multiplication of $d\times d$ matrices to take $O(d^\omega)$ field operations: %
for example, the standard method gives $\omega=3$. 
For sufficiently large $d$ (depending on the field size)
Magma~\cite{Magma07} uses the algorithm of \cite{Strassen} with 
$\omega = \log_2 7+\epsilon$ for any $\epsilon>0$: this gives a noticeable practical, as well as a theoretical, improvement.  

A key algorithmic problem for classical groups is the construction of isometries between classical forms.  We give a new method that is asymptotically faster than the method given in \cite{HRD}. 
\begin{thm}\label{thm:fix_form}
Suppose we have two nondegenerate symplectic, unitary, or quadratic forms on the space $V=(\F_{q^u})^d$.
We can determine if they are isometric, and find a canonical isometry between them, with
a Las Vegas algorithm taking $O(d^\omega + d^2\log^2 q)$ field operations.
\end{thm}

We now state our main theorem. 

\begin{thm}\label{thm:quot}
Let $\Omega \leq \gl_d(q^u)$ be a quasisimple classical group fixing a known classical form $F$, let  $\Delta = \mathrm{N}_{\gl_d(q^u)}(\Omega)$, and let $G=\Delta/\Omega$.
\begin{enumerate}
\item\label{thm:quot:pres} There is a deterministic algorithm which, on input $F$, constructs a finite presentation $P_1$ for $G$ in $O(\log^2 q)$ field operations. There is a Las Vegas algorithm which constructs the image under the homomorphism $\Delta \to P_1$ of $g \in \Delta$ in $O(d^\omega + d^2\log^2 q)$ field operations. 
\item\label{thm:quot:pc} There is a deterministic algorithm which, on input $F$, constructs a power-conjugate presentation $P_2$ for $G$ with at most $3$ generators and $6$ relations in $O(\log^2 q)$ field operations. There is a Las Vegas algorithm which constructs the image under the homomorphism $\Delta \to P_2$ of $g \in \Delta$ in $O(d^\omega + d^2\log^2 q)$ field operations, plus at most two calls to a discrete logarithm oracle for $\BBF_{q^2}$. 
\item\label{thm:quot:coset}  There is a Las Vegas algorithm which, on input $F$ and an element $g \in \Delta$, constructs a canonical representative of the coset $\Omega g$ in   $O(d^\omega +d^2\log^2 q)$ field operations.
\end{enumerate}
\end{thm}

By the \emph{type} of the form we mean one of: unitary, symplectic, orthogonal type $+$, orthgonal type $-$, orthogonal odd dimension. 
In Section~\ref{sec:forms} we define our canonical forms, and present algorithms for forms and classical groups, including proving Theorem~\ref{thm:fix_form}. 
In Section~\ref{sec:pres} we prove Theorem~\ref{thm:quot}. 
In Section~\ref{sect:app} we present some applications, before concluding in
Section~\ref{sect:timings} with some  data on our implementations: our algorithms are now part of the standard release of Magma.
The timings for our algorithms depend on the type of the form -- in Theorems~\ref{thm:fix_form} and~\ref{thm:quot} we have given worst-case timings, but more detailed results are given below.

\section{Groups and forms} \label{sec:forms}
In this section, we introduce some algorithms for classical forms and classical groups.
We require that the output of each algorithm %
 be \emph{canonical}:  
for fixed input, every call to the algorithm gives the same output, even if the algorithm is randomised.

\subsection{Fields}\label{subsect:fields}
Let $p$ be a prime and let $q$ be a power of $p$.
As is standard, we assume that $\F_{q}$ is constructed 
by adjoining a canonical root $\prim$ of the Conway polynomial \cite{ModAtlas} 
to the prime field $\F_p$, so that $\prim$ is the canonical primitive element of $\F_q$.
See \cite{LuebeckWebpage} for a current list of the fields for which this assumption is valid. 
We let $\primtwo$ be the canonical primitive element of $\F_{q^2}$, and recall that 
$\xi =\zeta^{q+1}$. 
Given a nonzero %
$\alpha\in\F_q$, the \emph{discrete logarithm} $\log_\xi(\alpha)$ is the unique $i=0,1,\dots,q-2$ %
such that $\alpha=\xi^i$.
We %
now show how to find canonical solutions to various equations over $\F_q$ or~$\F_{q^2}$. 

The next result is the main source of randomisation in our algorithms.
\begin{thm}[{\cite[Theorem 8.12]{Geddes92}}] \label{quadratic}
A root in $\F_{q^2}$ for a quadratic polynomial with coefficients in $\F_q$ can be found by a Las Vegas algorithm in $O(\log q)$ field operations.
\end{thm}
\noindent 
Let $\F_q^\times$ denote the multiplicative group of $\F_q$ and let
$\F_q^{\times 2}$ denote the set of squares in $\F_q^\times$.
Every element of $\F_{q^2}$ can be written as  
$a_0 + a_1\zeta + \dots + a_{m-1}\zeta^{m-1}$, where $p^m=q^2$ and 
$a_i \in \{0,\dots,p-1\}$.  
Lexicographically ordering the coefficients induces an ordering on $\F_{q^2}$. 
We fix a \emph{canonical root} of a quadratic equation by taking the smallest root with respect to this %
 ordering on $\F_{q^2}$.
Hence for $\alpha\in\F_q$ we can find a \emph{canonical square root} 
$\sqrt{\alpha}\in\F_{q^2}$. %
For $q$ even, $\alpha$ has a unique square root, equal to $\alpha^{q/2}$, so $\sqrt{\alpha}$ can be computed by a deterministic algorithm in $O(\log q)$ field operations. For $\alpha \in \F_q^\times$ with $q$ odd, define $\iota(\alpha) = 0$ if $\alpha \in \F_q^{\times 2}$ and $\iota(\alpha) = 1$ otherwise. Since $\iota(\alpha) = 0$ if and only if $\alpha^{(q-1)/2} = 1$, there is a deterministic algorithm to determine $\iota(\alpha)$  which takes  $O(\log q)$ field operations.

Canonical solutions for trace and norm equations are needed for the unitary groups.
\begin{prop}\label{prop:norm}
Let $\alpha \in \F_q^\times$. There is a deterministic algorithm to find a canonical solution $\eta \in \F_{q^2}$ to the trace equation $\eta + \eta^q = \alpha$ which takes $O(1)$ field operations if $q$ is odd, and $O(\log q)$ otherwise. There is a Las Vegas algorithm to find a canonical solution $\eta \in \F_{q^2}$ of the norm equation $\eta^{q+1}=\alpha$ which takes $O(\log q +\log^2 p)$ field operations.
\end{prop}
\begin{proof}
For the trace equation with $q$ odd, $\eta=\alpha/2$.  
Otherwise, use the fact that $\alpha\mapsto \alpha^q$ is an $\BBF_{q}$-linear map.
After we evaluate this map on an $\BBF_{q}$-basis of $\BBF_{q^2}$ deterministically in $O(\log q)$ field operations, the problem is reduced to two dimensional system of linear equations over $\BBF_{q}$.
Since $\eta$ exists by \cite[Theorem~6.3]{Lang93Book}, it can now be found by linear algebra.

We construct a solution to the norm equation in three cases.
If $\alpha \in \F_q^{\times 2}$, let $\eta:=\sqrt{\alpha}$, then $\eta^{q+1}=\eta^2=\alpha$. %
If $\alpha \not\in \F_q^{\times 2}$ and $q\equiv 1\pmod{4}$, then 
$-1 \in \BBF_q^{\times 2}$, so $-\alpha \not\in \BBF_q^{\times 2}$.
Hence the polynomial $X^2+\alpha$ is irreducible over $\F_q$, and  its roots in $\BBF_{q^2}$ have norm~$\alpha$, which can be found by Theorem~\ref{quadratic}.
If $\alpha \not\in \F_q^{\times 2}$ and $q\equiv 3\pmod{4}$, then $-\alpha \in \F_q^{\times 2}$. Let $\beta = \sqrt{-\alpha}$ and write $p+1=2^ms$ for $s$ odd.
Calculate $c\in\F_p$ in $O(\log^2 p)$ field operations by
$$ 
  c_1 := 0; \quad \quad 
  c_{i+1} := \left(\frac{c_i+1}{2}\right)^{\frac{p+1}4} \ 
    (i=1,\dots,m-2); \quad \quad 
  c := \left(\frac{c_{m-1}-1}{2}\right)^{\frac{p+1}4}.
$$ 
By \cite{BlakeGaoMullin93}, the polynomial $g(X)=X^2-2cX-1$ is
irreducible over~$\F_q$. 
Hence $-\alpha g(X/\beta)=X^2-2\beta c X+\alpha$ is also irreducible and its roots in $\BBF_{q^2}$ have norm~$\alpha$.
\end{proof}
The following elements are all used to compute with orthogonal groups.
\begin{prop}\label{prop:canelts}
$\;$
\begin{enumerate}
\item\label{thm:canelts:gammaodd} There is a deterministic algorithm to construct, on input an odd $q$, a canonical $\gamma\in\F_q^\times$ such that $\gamma$ and $1-4\gamma$ are nonsquare. The algorithm takes $O(\log q)$ field operations.
\item\label{thm:canelts:gammaeven} There is a deterministic algorithm to construct, on input an even $q$, a canonical $\gamma\in\F_q^\times$ such that $X^2+X+\gamma$ 
is irreducible over $\F_q$. The algorithm takes $O(\log^2 q)$ field operations.
\item\label{thm:canelts:nu} There is a deterministic algorithm to construct, on input an odd $q$, a canonical $\nu\in\F_q^\times$ such that $1 + \nu^2$ is nonsquare. The algorithm takes $O(\log q)$ field operations. 
\end{enumerate}
\end{prop}
\begin{proof}
For (\ref{thm:canelts:gammaodd}), note that $\zeta+\zeta^q\ne0$ (recall that $\zeta$ is the canonical primitive element %
 in $\F_{q^2}$), as otherwise 
$\zeta^{q-1}=-1=\zeta^{(q^2-1)/2}$.
Set $\gamma =\xi(\zeta+\zeta^q)^{-2}$, then $\gamma\in\F_q$ because $\gamma^q=\gamma$.
Also,  $\gamma \not\in \BBF_q^{\times 2}$  because $\xi \not\in \BBF_q^{\times 2}$. Finally, $1-4\gamma = %
(\zeta-\zeta^q)^2(\zeta+\zeta^q)^{-2} \not\in \BBF_{q}^{\times 2},$
since $(\zeta-\zeta^q)(\zeta+\zeta^q)^{-1} \not\in\F_{q}$.

For (\ref{thm:canelts:gammaeven}), let $q=2^m$.  If $m$ is odd, let $\gamma=1$. Otherwise, let  $m = 2^r s$ with %
 $s$ odd.
Define $a_i$ recursively:  let %
$a_0=1$,  and let %
$a_{i+1}$ be the canonical root of $X^2+X+a_i$ in $\F_q$.  
Define $\gamma$ to be the first $a_j$ for which $X^2+X+a_j$ is irreducible, if any.
Define  $T:\F_q\to\F_q$ by $T(x)=x^2+x$, and note that 
$T(a_i)= a_i^2 + a_i = a_{i-1}$ for $i\ge1$.
It is easy to show that 
$T^{2^i}(x)=x^{2^{2^i}}+x$ for all $i$.
Now suppose $a =a_{2^r+1}\in\F_q$ exists.
Then $T^{2^r+1}(a)=1$, so $T^{2^{r+1}}(a) = T^{2^{r+1}-2^r-1}(1)=0$, and so
$a^{2^{2^{r+1}}}=a$.
Hence $a \in \F_{2^{2^{r+1}}}$, which intersects $\F_q$ in 
$\F_{2^{2^r}} $.
This implies that $a^{2^{2^r}}=a$, so 
$T^{2^r}(a)=0$, which contradicts $T^{2^r+1}(a)=1$. Therefore 
$j\le2^r\le\log q$.

For (\ref{thm:canelts:nu}), note that $4 \zeta^{q+1}(\zeta - \zeta^q)^{-2} \in \BBF_q^{\times 2}$. Let $\nu  = 2 \zeta^{(q+1)/2}(\zeta - \zeta^q)^{-1}\in \BBF_q$ be its
square root, then $1 + \nu^2 \not\in \BBF_q^{\times 2}$.   
\end{proof}

\subsection{Forms and Isometries} \label{subsec:forms}
In this subsection, we define our canonical forms, and 
present algorithms to construct %
isometries and similarities %
 between forms. 

Let $V=(\F_{q^u})^d$ and let $v_1, \ldots, v_d$ be the basis of $V$ with $(v_i)_j=1$ if $i=j$ and $0$ otherwise. %
By $\diag(a_1, a_2, \ldots, a_d)$ we mean the $d \times d$  matrix
with entry $a_i$ in position $(i, i)$ and $0$ elsewhere. 
By $\antidiag(a_1, a_2, \ldots, a_d)$ we mean the $d \times d$ 
matrix with entry $a_i$ in position $(i, d-i+1)$ and $0$
elsewhere. By $A \oplus B$ we mean a block diagonal matrix, with blocks $A$ and $B$ along the main diagonal and $0$ elsewhere. 
We denote the transpose of $A$ by $A^\Tr$.

The following results are standard and can 
be found in  \cite[Chapter 16]{Burgisser90}.
\begin{thm} \label{mat_complexity}
There are deterministic algorithms to find the row echelon form, the rank, the nullspace, or the determinant of a $d \times d$ matrix over $\F_q$. Each algorithm requires $O(d^\omega)$ field operations. 
\end{thm}

We refer to \cite{Taylor} or \cite{Grove02Book} for basic terminology on classical forms.
We fix the following notation:
either $\beta$ is a nondegenerate symplectic or unitary form over $V$;  or
$Q$ is a nondegenerate quadratic form over $V$ and $\beta$ is its polar form, so that $2Q(v)=\beta(v,v)$.
A vector $v$ is \emph{isotropic} if $\beta(v, v)= 0$ and 
\emph{singular} if $Q(v) = 0$: note that if $q$ is even and the form is quadratic
then there can exist vectors that are isotropic but nonsingular. A vector is \emph{anisotropic} if $Q(v) \neq 0$. 
The matrix of $\beta$ is $F = (\beta(v_i, v_j))_{d \times d}$, and satisfies
$\beta(u,v) = uFv^{\sigma\Tr}$, where $\sigma$ is the field automorphism $x \mapsto x^q$ (nontrivial only in the unitary case). %
The matrix of $Q$ is the upper triangular matrix $M=(m_{ij})_{d \times d}$ such that $Q(v) = vMv^\Tr$
for $v = (a_1, \ldots, a_d)$.
If $\beta$ is the polar form of $Q$, then
$F=M+M^{\text{Tr}}$ and $F$ determines $M$ if and only if $q$ is odd.
Forms $\beta_1$ and $\beta_2$ (or $Q_1$ and $Q_2$) 
are \emph{isometric} if there exists an $A\in\text{GL}_d(q^u)$ such that
$\beta_1(u,v)=\beta_2(uA,vA)$ for all $u,v\in V$ (respectively, such that $Q_1(v) = Q_2(vA)$ for all $v \in V$). %
Forms $\beta_1$ and $\beta_2$ (or $Q_1$ and $Q_2$)  %
are \emph{similar}
if there exists a $\lambda \in \F_{q^u}^\times$ such that  $\beta_1$ is isometric to $\lambda\beta_2$ (respectively, such that $Q_1$ is isometric to $\lambda Q_2$).

\begin{df}[Canonical classical forms]\label{def:standard_forms} 
We define the following canonical forms: \\
\noindent {\bf Symplectic or even dimension unitary:} $d = 2m$ and  
  $V$ has basis $(e_1, \ldots, e_m, f_m, \ldots, f_1)$ 
  with $\beta(e_i,e_j) = \beta(f_i,f_j) = 0$, 
  $\beta(e_i,f_j) = \delta_{ij}$.\\
{\bf Unitary, odd dimension:}  ${d = 2m+1}$ and  
  $V$ has basis $(e_1, \ldots, e_m, x, f_m, \ldots, f_1)$ 
  with $\beta(e_i,e_j)$   \\
 $= \beta(f_i,f_j) %
= \beta(e_i,x) = \beta(f_i,x)= 0$, 
  $\beta(e_i,f_j) = \delta_{ij}$, $\beta(x,x)=1$.\\
{\bf Orthogonal, $\circ$ type:}  ${d = 2m+1}$ and  
  $V$ has basis $(e_1, \ldots, e_m, x, f_m, \ldots, f_1)$ 
  with $Q^\circ(e_i) = Q^\circ(f_i) =  \beta^\circ(e_i, e_j) = \beta^\circ(f_i, f_j) = %
\beta^\circ(e_i,x) = 
  \beta^\circ(f_i,x)= 0$, 
  $\beta^\circ(e_i,f_j) = \delta_{ij}$, $Q(x)=1$.\\
{\bf Orthogonal, $\mathbf{+}$ type:} ${d = 2m}$ and
   $V$ has basis $(e_1, \ldots, e_m, f_m, \ldots, f_1)$ with 
  $Q^+(e_i) = Q^+(f_j) = \beta^+(e_i, e_j) = \beta^+(f_i, f_j)  %
= 0$ and $\beta^+(e_i, f_j) = \delta_{ij}$.\\
{\bf Orthogonal, $\mathbf{-}$ type:} ${d = 2m+2}$ and
  $V$ has basis $(e_1, \ldots, e_{m}, x, y, f_m, \ldots, f_1)$ with 
  $Q^-(e_i) = Q^-(f_j) = \beta^-(e_i, e_j) = \beta^-(f_i, f_j) %
= 0$, $\beta^-(e_i, f_j) = \delta_{ij}$, $\beta^-(a, b) = 0$ for 
  $a \in\{e_i, f_j\}$, $b \in \{x,y\}$, $Q^-(x)=\beta^-(x, y)=1$, $Q^-(y)= \gamma$, where $\gamma$ is as in Proposition~\ref{prop:canelts}. 
\end{df}

It is well known (see for instance \cite{Taylor}) that
every nondegenerate quadratic, symplectic or unitary form over a finite field is similar to exactly one of the forms given in Definition~\ref{def:standard_forms}.
For odd dimension and characteristic, the two isometry classes of quadratic forms are similar.  
Otherwise, forms are similar if and only if they are isometric.
The \emph{discriminant} of $Q$ is $\iota(\det(F))$. 
Two quadratic forms 
are isometric if and only if they have the same discriminant.

The following will be needed for constructing isometries and 
coset representatives.  Unitary forms have an anisotropic
vector whenever they are not identically zero, and quadratic forms have a nonsingular 
vector whenever they are not identically zero. 
However, symmetric forms may not have an anisotropic vector
in even characteristic. 
\begin{lem}\label{lem:vectors}
There is a deterministic algorithm which, on input a nonzero quadratic form, finds a canonical nonsingular vector $v$  in $O(d^2)$ field operations. 
There is a deterministic algorithm which, on input a nonzero quadratic form in odd characteristic or a nonzero unitary form, finds a canonical anisotropic vector $w$ in $O(d^2)$ field operations. 
There is a Las Vegas algorithm which, on input a nondegenerate quadratic form $Q$ with $q$ odd and $d \geq 2$, finds canonical nonsingular vectors $u_1, u_2$ such that $\iota(Q(u_1)) = 0$ and $\iota(Q(u_2)) = 1$ in  $O(d^2 + \log q)$ field operations. 
\end{lem}

\begin{proof}
We first discuss finding $v$ or $w$. To find $v$, let $M=(m_{ij})$ be the matrix of the quadratic form. To find $w$, let $M$ be the matrix of the polar form of $Q$ or of the unitary form.
To find $v$ or $w$, now look for the smallest $i$ such that $m_{ii} \neq 0$. If $i$ exists, take $v = v_i$ or $w = v_i$. 
If none exists, let $(i,j)$ be lexicographically minimal subject to $m_{ij} \neq 0$. Let
$v = v_i + v_j$, and in the quadratic case let $w = v_i +v_j$ also.  If $M$ is unitary, let $w = v_i + \zeta v_j$, so that $\beta(v,v)=\zeta+\zeta^q$, which is nonzero as observed in the proof of Proposition~\ref{prop:canelts}(1). 

To find $u_1$ and $u_2$, first choose $v_1$ nonsingular as above. Compute  $v_1^\perp$ as the nullspace of the column vector $Fv_1^\Tr$ in $O(d^2)$ field operations, then recursively choose nonsingular $v_2 \in v_1^\perp$: note that $v_2 \not\in \langle v_1 \rangle$ as $v_1$ is nonsingular. 
If possible, take $u_1=v_i$ for square $Q(v_i)$ and $u_2=v_j$ for nonsquare $Q(v_j)$. 
If this is not possible, then either the $Q(v_i)$ are both square, or both are nonsquare.
Let $w=v_1+\nu \sqrt{Q(v_1)/Q(v_2)}v_2$, where $\nu$ is as in Proposition~\ref{prop:canelts}. Then  $Q(w)= (1+\nu^2)Q(v_1)$ and hence $\iota(Q(w)) = 1$ if and only if $\iota(Q(v_1)) = 0$, so let $u_1$ be one of $w$ or $v_1$ and let $u_2$ be the other. 
\end{proof}

Next we present the main technical ingredient of our isometry construction algorithm. We deal uniformly with symplectic, unitary and symmetric bilinear forms, and refer to the symplectic case as \emph{case {\sf S}}. 
 We define the \emph{initial $k$-block} of a matrix $X$ to be the matrix consisting  of the first $k$ columns of the first $k$ rows of $X$. 
For a matrix over $\F_{q^2}$, the map $\sigma$ is the $q$th power map on matrix entries and so %
the application of $\sigma$ takes $O(\log q)$ field operations for each entry. For a matrix $X$, we write $X^\ast$ for $-X^\Tr$ in case {\sf S},  for $X^{\sigma \Tr}$ in the unitary case, and for $X^\Tr$ in the orthogonal case. Furthermore, we write $X^\dag$ for $X^\Tr$ in case {\sf S} and for $X^\ast$ otherwise. 
Let $a = \log q$ in the unitary case and $0$ otherwise. 
If $SAS^\dag =B$ we say that $S$ \emph{transforms} $A$ to $B$.
Note that we do not assume that our forms are nondegenerate, so  symplectic forms can have odd dimension.

\begin{thm}[Diagonalise forms]\label{thm:diag}
Let $A$ be the matrix of a (possibly degenerate) symmetric, unitary, or symplectic form over $\BBF_{q^u}$, where if $q$ is even then the form is unitary or symplectic.   There is a deterministic algorithm which, on input $A$, constructs a canonical $S \in \gl_d(q^u)$ such that $SAS^\dag$ is
diagonal, or block diagonal with blocks of size %
at most $2$ in case {\sf S}. The algorithm takes  $O(d^\omega + d^2a)$ field operations, where $a$ is $\log q$ in the unitary case and $0$ otherwise.
\end{thm}

We prove the result via a sequence of lemmas.   

\begin{lem}\label{L1} Let $A$ be a matrix of the form
$$\left( \begin{array}{ccc}
A_1  & 0 & A_2 \\
0 & 0 & A_3 \\
A_2^\ast & A_3^\ast & A_4 \end{array} \right),$$ where $A_1 \in \gl_{k}(q^u)$ for $1 \leq k \leq d-1$ (with $k$ even in case {\sf S}) and $A_3$ has $0 \leq s <  d-k$ rows. There is a deterministic algorithm which, on input $A$, constructs a canonical $S \in \gl_d(q^u)$ such that $$SAS^{\dag} = 
A_1 
\oplus \left( \begin{array}{cc}
 0 & A_3 \\
 A_3^\ast & A_5 \end{array} \right).$$
The algorithm takes $O(d^\omega + d^2a)$ field operations.
\end{lem}
\begin{proof}
Let $S = \left( \begin{smallmatrix} I_{k} & 0 & 0 \\
0 & I_s & 0 \\
-A_2^{\ast} %
 A_1^{-1} & 0 & I_{d- k - s} \end{smallmatrix} \right)$.
\end{proof}

\begin{lem}\label{L2} There is a deterministic algorithm which, on input  $A \neq 0$, constructs a canonical
 $S \in \gl_d(q^u)$ such that $SAS^{\dag} = 
 A_1 \oplus 0$ with $A_1 \in \gl_{k}(q^u)$ for some  $1 \leq k \leq d$ (with $k$ even in case {\sf S}). The algorithm takes $O(d^\omega)$ field operations.
\end{lem}
\begin{proof}
Let $S \in \gl_d(q^u)$ be such that $SA$ is in row echelon form,  constructed in $O(d^\omega)$ field operations by Theorem~\ref{mat_complexity}. Then $$SAS^{\dag} = \left( \begin{array}{c} X \\ 0 \end{array} \right) S^{\dag} = Y$$ for some  matrix $X_{k \times d}$ with full row rank. Now, $Y$ has its final $d-k$ rows all zero, and $Y = Y^\ast$. Thus the final $d-k$ columns of $Y$ are all zero, and the initial $k$-block of $Y$ is in $\gl_k(q^u)$. 
\end{proof}

\begin{lem}\label{L3}
Let $d  \equiv 0 \bmod 4$ in case {\sf S}, and let $d$ be even otherwise. There is a deterministic algorithm which,  on input
$$A = \left( \begin{array}{cc} 
0 & A_1 \\
A_1^{\ast} & A_2 \end{array} \right)$$ with $A_1 \in \gl_{d/2}(q^u)$, constructs a canonical $S \in \gl_d(q^u)$ such that the initial $(d/2)$-block of $SAS^{\dag}$ is invertible. The algorithm takes $O(d^\omega + d^2a)$ field operations.
\end{lem}
\begin{proof}
First use Lemma~\ref{L2} to construct $U \in \gl_{d/2}(q^u)$ in $O(d^\omega)$ such that $UA_2U^{\dag} = A_3 \oplus 0$, 
with $A_3 \in \gl_{k}(q^u)$ for some $k \leq d/2$ (and $k$ even in  case {\sf S}). Construct $S_1 = (A_1 U^{\dag})^{-1} \oplus U$ in $O(d^\omega + ad^2)$ field operations, then $$B:= S_1A S_1^{\dag} = \left( \begin{array}{cc}
0 & I_{d/2} \\
I_{d/2}^\ast & 
A_3 \oplus 0 \end{array} \right).$$ It is now routine to construct a canonical $S_2$ such that $S_2 B S_2^\dag$ has invertible initial $(d/2)$-block. 
\end{proof}

\begin{lem}\label{L4}
Let $l$ with $1 \leq l \leq d-1$ be given, with $l$ even in case {\sf S}.  There is a deterministic algorithm which, on input an invertible matrix  $A$, constructs a canonical $S \in \gl_d(q^u)$ such that the initial $l$-block of $SAS^{\dag}$ is invertible. The algorithm takes $O(d^\omega + d^2a)$  field operations.
\end{lem}
\begin{proof}
If $l > 1$ then first construct a canonical permutation matrix $S_1$ transforming $A$ to a matrix $B$ whose initial $l$-block is not identically zero.
If $l = 1$ and $a_{11} = 0$  then construct a canonical anisotropic vector $v$ in $O(d^2)$ field operations, by Lemma~\ref{lem:vectors}, and let $B$ be the form resulting from  swapping this $v$ with $v_1$. 
 Let $$B = \left( \begin{array}{cc}
B_1 & B_2 \\
B_2^{\ast} & B_3 \end{array} \right),$$
where $B_1$ is $l \times l$.  If $B_1$ is invertible, we are done. Otherwise, 
construct a matrix $S_2$ such that 
$$C:= S_2 B S_2^{\dag} = \left( \begin{array}{cc}
C_1 \oplus 0 & C_2 \\
C_2^{\ast} & B_3 \end{array} \right),$$
where $C_1  = C_1^\ast \in \gl_{k}(q^u)$ for some $k < l$ (with $k$ even in case {\sf S}). The matrix $C$ can be computed in $O(d^\omega + ad^2)$ field operations by Lemma~\ref{L2}.  Since $C_1$ is invertible, by Lemma~\ref{L1} in $O(d^\omega + ad^2)$ field operations we construct a matrix $S_3$ such that 
$$D:= S_3 C S_3^{\dag} = C_1 \oplus \left( \begin{array}{cc}
 0 & D_1 \\
D_1^{\ast} & D_2 \end{array} \right),$$
where $D_1$ is $(l-k) \times (d-l)$. The fact that $A$ and $C_1$ are both invertible implies that $D_1$ has full row rank, so construct a matrix $P \in \gl_{d-l}(q^u)$ in $O(d^\omega)$ field operations such that $D_1 P = (E_1 \ E_2)$ with $E_1 \in \gl_{l-k}(q^u)$. Let $S_4:= I_{l} \oplus P^{\dag}$. Then 
$$E:= S_4DS_4^{\dag} = 
C_1\oplus\left( \begin{array}{cccc}
0 & E_1 & E_2 \\
E_1^{\ast} & E_3 & E_4 \\
E_2^{\ast} & E_4^{\ast} & E_5 \end{array} \right),$$ where $E_3$ is $(l-k) \times (l-k)$. By Lemma~\ref{L3}, in $O(d^\omega + ad^2)$ field operations we can construct a $2(l-k) \times 2(l-k)$ matrix $M$ such that 
$$M \left( \begin{array}{cc} 0 & E_1 \\
E_1^\ast & E_3 \end{array} \right) M^\dag$$
has initial $(l-k)$-block invertible. Let 
 $S_5 = I_{k} \oplus M \oplus I_{d - 2l + k}$, then
$S_5 E S_5^{\dag}$ %
has invertible initial $l$-block. 
\end{proof}

\begin{proof}[Proof of Theorem~\ref{thm:diag}] 
If $A$ is identically zero, there is nothing to do. Otherwise, by %
 Lemma~\ref{L2}, in $O(d^\omega + d^2a)$ field operations we can transform $A$ to $S_1AS_1^\dag = A_1 \oplus 0$ with $A_1 \in \gl_{r}(q^u)$ for some $r \leq d$, with $r$ even in case {\sf S}.
 Then by Lemma~\ref{L4}, in $O(d^\omega + d^2a)$ field operations we can construct a matrix $S_2$ transforming $A_1$ to a matrix $A_2$ whose initial $k$-block $B_1$ is invertible, where $k = 2\lfloor r/4 \rfloor$ in case {\sf S} and $k = \lfloor r/2 \rfloor$ otherwise. Now by Lemma~\ref{L1}, in $O(d^\omega + d^2a)$ field operations we can construct a matrix $S_3$ transforming $A_2$ to $B_1 \oplus C_1$, where $C_1 = C_1^\ast \in \gl_{r %
-k}(q^u)$. We now recurse on $B_1$ and $C_1$, stopping when we reach $2 \times 2$ matrices in case {\sf S} or $1 \times 1$ matrices otherwise. The whole process completes in $O(d^\omega + d^2a)$ field operations and produces canonical matrices at each step. 
\end{proof}

We remark that the symmetric case of the above theorem is proved in \cite[Theorem 16.25]{Burgisser90}, although we correct several minor errors in the proof. 

\begin{thm}[Transform forms]\label{thm:fix_form_detail}
Suppose we have two nondegenerate symplectic, unitary, or quadratic forms on the space $V=(\F_{q^u})^d$.
We can determine if they are isometric, and find a canonical isometry between them, in $O(C)$ field operations, where $C$ is given in Table~\ref{tab:fix}.
The algorithm used is deterministic for symplectic forms; otherwise it is Las Vegas.
\end{thm}

\begin{table} 
\caption{Complexity for transforming forms}\label{tab:fix}
\begin{tabular}{l||l}
Form type			& $C$ \\\hline\hline
Symplectic			& $d^\omega$ \\
Unitary				& $d^\omega +d^2\log q + d\log^2 p$ \\
Quadratic,	 $q$ odd	& $d^\omega +d\log q$ \\
Quadratic,	 $q$ even	& $d^\omega +d\log q + \log^2 q$
\end{tabular}
\end{table}

\begin{proof}
Note that it is enough to find an isometry or similarity from a given form to some fixed form. For quadratic forms we work at least initially with the polar form. 

If the form is of unitary type, or the polar form of a quadratic form in odd characteristic, then use Theorem~\ref{thm:diag} to diagonalise the matrix of the form to $\diag(a_1, \ldots, a_d)$. In case {\sf S} (resp.\ the form is the polar form of a quadratic form in even characteristic), then transform its matrix to a block diagonal matrix with $2 \times 2$ (and $1 \times 1$) blocks.
 
In the symplectic case, each $2\times2$ block is equal to $\antidiag(a, -a)$ for some $a \in \BBF_q^\times$. This is transformed to $\antidiag(1, -1)$  by $\diag(a^{-1}, 1)$.

In the unitary case, the form is transformed to $I_d$
 by $\diag(\alpha_1, \ldots, \alpha_d)$, where $\alpha_i$ is a canonical solution to $\alpha_i^{q+1} = a_i^{-1}$, using Proposition~\ref{prop:norm}. %

In the orthogonal case for $q$ odd, if $d$ is odd and the discriminant is nonsquare then let $\alpha$ be the first nonsquare entry, and multiply all entries by $\alpha^{-1}$ (we produce a similarity since $\alpha\ne1$).  In all orthogonal cases now transform all the square entries
$a_i$ to $1$ by $\sqrt{a_i}^{-1}$ and the nonsquare entries $a_i$ to the first nonsquare entry, $\mu$, by $\sqrt{\mu/a_i}$.  The entries $\mu$ are then changed in pairs to $\mu(1+\nu^2)$, using the fact that
$\left(\begin{smallmatrix}
1 & \nu \\
-\nu & 1 \end{smallmatrix} \right)\left(\begin{smallmatrix}
1 & \nu \\
-\nu & 1 \end{smallmatrix} \right)^{\text{Tr}} = (1+\nu^2)I_2$, 
where $\nu$ is as in Proposition~\ref{prop:canelts}.
Each entry $\mu(1+\nu^2)$ can now be changed to $1$, since $\mu(1+\nu^2)\in\F^{\times2}$. If there is a single nonsquare entry remaining (so that $d$ is even) then this is moved to the first row and transformed to $\prim$. 

In the orthogonal case for $q$ even, the way that we have  transformed the polar form matrix $F$ also makes the matrix $M$ of the quadratic form block diagonal with blocks of size at most 2 (since $F$ and $M$ are identical above the diagonal). We now work with $M$. 
Since every element of $\F_q$ has a square root, we can convert
every block in $M$ to one of the forms $(1)$,
$\left(\begin{smallmatrix}1&a\\0&1\end{smallmatrix}\right)$, or
$\left(\begin{smallmatrix}0&1\\0&0\end{smallmatrix}\right)$.
Note that a summand
$\left(\begin{smallmatrix}1&a\\0&1\end{smallmatrix}\right)$ must have
$a\ne0$, otherwise it would be degenerate and so $Q$ would also be
degenerate.  This also shows that there is at most one summand $(1)$.

Now consider a subform whose matrix is a pair of $2\times2$ blocks:
$\left(\begin{smallmatrix}1&a\\0&1\end{smallmatrix}\right)\oplus
\left(\begin{smallmatrix}1&b\\0&1\end{smallmatrix}\right)$ with
respect to the basis $u_1,u_2,u_3,u_4$.
Changing to the basis
$u_1+u_3,(u_1+u_4)/b,u_1,bu_2+a(u_3+u_4)$, we get the form with matrix
$\left(\begin{smallmatrix}0&1\\0&0\end{smallmatrix}\right)\oplus
\left(\begin{smallmatrix}1&ab\\0&b(a^2+b)\end{smallmatrix}\right)$.
The second block can now be converted to $\left(\begin{smallmatrix}0&1\\0&0\end{smallmatrix}\right)$ or $\left(\begin{smallmatrix}1&a'\\0&1\end{smallmatrix}\right)$ for some $a'\ne0$ as above.

So we eventually get a direct sum of copies of
$\left(\begin{smallmatrix}0&1\\0&0\end{smallmatrix}\right)$ together
with at most one block of the form $(1)$ or
$\left(\begin{smallmatrix}1&a\\0&1\end{smallmatrix}\right)$.
If the polynomial $X^2+X+a$ has a solution in $\F_q$, then
$\left(\begin{smallmatrix}1&a\\0&1\end{smallmatrix}\right)$ can be
transformed to $\left(\begin{smallmatrix}0&1\\0&0\end{smallmatrix}\right)$,
and otherwise it can be transformed to
$\left(\begin{smallmatrix}1&1\\0&\gamma\end{smallmatrix}\right)$.
So %
we are done.
\end{proof}

Theorem~\ref{thm:fix_form} is just a simplified version of this result.
Note that Theorems~\ref{thm:fix_form} and~\ref{thm:fix_form_detail} apply unchanged to computing similarities rather than isometries.

\subsection{Groups}\label{sec:maps}
 
Suppose  $\beta$ (or $Q$) is a nondegenerate form,  as in the previous subsection.
Then $\Delta:= \mathrm{N}_{\gl_d(q^u)}(\Omega)$ consists of all similarities of the form with itself.
The invariant group $I$ consists of all isometries.
We use notation from \cite{KL90} for classical groups.
For example, if $\beta$ is a symplectic form, then $\Delta=\text{CSp}_d(q,\beta)$;  if $\beta$ is the canonical symplectic form, then we abbreviate this to $\text{CSp}_d(q)$.

Define $\tau:\Delta\to \F_{q^u}$ by $\beta(ux,vx)=\tau(x) \beta(u,v)$ for all $u,v\in V$. 
It is well known (see for example \cite[Lemma~2.1.2]{KL90}) that $\tau$ is a homomorphism with kernel $I$.

\begin{lem} \label{lem:tau}
There is a deterministic algorithm which, on input $g \in \Delta$ and the matrix $F$ of $\beta$, computes $\tau(g)$ in $O(d^2)$ field operations.
\end{lem}
\begin{proof}
Find $w$ such that $w F v_1^\Tr \neq 0$ in $O(d)$ field operations.
Then $\tau(g)$ is $\beta(wg,v_1g)/\beta(w,v_1)$.
\end{proof}

For quadratic forms, the spinor norm is an epimorphism from the general orthogonal group $I=\orth_d(q, Q)$ to $\F_2^+$.
\begin{df}[Spinor norm] $\phantom{x}$ Let $g \in \gl(d, q)$ preserve the form $Q$.
\begin{enumerate}
\item For $q$ odd, let $U\le V$ be the image of $I_d-g$ and define the
  bilinear form 
  $\chi$ on $U$ by $\chi(u,v)=2\beta(w,v)$ where $w(I_d-g)=u$.
  The \emph{spinor norm} of $g$ is $\spin(g)=\iota(\det(\chi))$.
\item
  For $q$ even, the \emph{spinor norm} of $g$ is $\spin(g)=\rank(I_d+g) \bmod 2$.
\end{enumerate}
\end{df}
\noindent
Our definition for odd $q$ is from \cite{Taylor}, except for the factor of two which we include so the values of the spinor norm agree with \cite[p.29]{KL90}. 
We follow  \cite[Proposition~2.5.7]{KL90} and define $\Omega_d(q, Q):= \so_d(q,Q) \cap \ker(\spin)$.
What we call the spinor norm for even $q$ is called the Dickson invariant by some authors.
\begin{thm}\label{thm:spin}
There is a deterministic algorithm that, on input $g \in \orth_d(q, Q)$, computes  $\spin(g)$. If $q$ is even then the algorithm takes  %
 $O(d^\omega)$ field operations,  otherwise it takes
 $O(d^\omega + \log q)$ field operations. 
\end{thm}
\begin{proof}
If $q$ is even, apply Theorem~\ref{mat_complexity}.
If $q$ is odd,  
compute the nullspace $N$ of $a:= I_d - g$ and find a matrix  $M$ whose rows are a basis to a complement of $N$ in $O(d^\omega)$ field operations. Then the rows of $Ma$ are a basis for the image of $a$. 
Calculate the form $\chi_g$  on $Ma$ as $S = 2MF(Ma)^{\Tr}$ in $O(d^\omega)$ field operations. 
Finally, find $\iota(\det S)$. 
\end{proof}

We finish this section with a discussion of reflections. 
Let  $v \in V$ be nonsingular, so that $Q(v) \neq 0$. 
The \emph{reflection} in $v$ is the map $\refl_v:V\to V,$ 
$u \mapsto u-\beta(u,v)v/Q(v)$.  

\begin{lem}\label{lem:refl} 
Let $Q$ be nondegenerate with polar form $F$, and let $u, v \in V$ be nonsingular. 
\begin{enumerate}
\item\label{lem:refl:det} All reflections are elements of $\orth_d(q,Q)$, and have determinant $-1$ and order $2$.
\item\label{lem:refl:spine} For $q$ even, $\spin(\refl_v)=1$.
\item\label{lem:refl:spino} For $q$ odd,  $\spin(\refl_v)=\iota(\beta(v, v))$. 
\item\label{lem:refl:coset} For $q$ odd,
  $\Omega_d(q,Q)\refl_u = \Omega_d(q,Q)\refl_v$ if and only if
  $\iota(\beta(u, u))=\iota(\beta(v, v)).$
\end{enumerate}
\end{lem}
\begin{proof}
Parts (\ref{lem:refl:det}) and (\ref{lem:refl:spine}) are well-known, and are easy exercises. 
For part (\ref{lem:refl:spino}), let $g=\refl_v$.
Then $(I_d-g)$ has image $\langle v \rangle$, and maps $v \mapsto 2v$, so the matrix of $\chi_g$ is $(\beta(v, v))_{1 \times 1}$. 
Part (\ref{lem:refl:coset}) follows from part (\ref{lem:refl:spino}) and the fact that $\spin$ is a homomorphism. 
\end{proof}

\begin{prop} \label{prop:reflections}
Let $Q$ be nondegenerate. 
For odd $q$ and $d \geq 2$, there is a Las Vegas algorithm that constructs canonical reflections $R_0,R_1$ with $\spin(R_i)=i$ in $O(d^2+\log q)$ field operations. 
For even $q$ and $d \geq 2$, a canonical reflection $R_0$ can be constructed deterministically in $O(d^2)$ field operations.
\end{prop}

\begin{proof}
For $q$ odd, by Lemma~\ref{lem:vectors} we can find canonical vectors $u_0,u_1$ with $\iota(Q(u_i))=i$.
Note that $u_iFv_j^{\Tr}$ can be computed in $O(d)$ 
field operations for each $j$, 
as $Fv_j$ is the $j$th row of $F$. Then row $j$ of $\refl_{u_i}$ is 
$v_j - (u_iFv_j^{\Tr})Q(u_i)^{-1}u_i$. The case $q$ even %
 is similar. 
\end{proof}

\section{Constructive homomorphisms}\label{sec:pres}
In this section, for each type of classical group, we construct the quotient of 
the conformal group $\Delta$ by the quasisimple group $\Omega$ as a presentation in two ways. 
The first presentation has $O(q)$ generators, and a word for the image of an element of $\Delta$ can be found in polynomial time.
The second presentation is polycyclic with at most four generators and at most six relations, but words for images can only be found using discrete logarithms. 
To our knowledge, for the orthogonal groups such 
presentations only exist in the literature for the projective groups \cite[Sections~2.5--2.8]{KL90}. 
Note that the first presentation has a constant number of generators and relations when considered as an FC-presentation in the sense of \cite{CohenHallerMurray08}.
We also compute canonical representatives for cosets of $\Omega$, which are needed for the conjugacy problem in Section~\ref{sect:app}.
Throughout this section we assume that $\Omega$ is quasisimple, which eliminates some small dimensional exceptional cases.

Our  main result in this section is the following theorem.

\begin{thm}\label{thm:quotdetail}
Let $\Omega \leq \gl_d(q^u)$ be a quasisimple classical group fixing a known classical form, let  $\Delta = \mathrm{N}_{\gl_d(q^u)}(\Omega)$
and let $G:= \Delta/\Omega$.
Let $X$ be the matrix tranforming the canonical form to the given form (Theorem \ref{thm:fix_form}).
Let $X_i$, $\sR_i$, and $C_i$ ($i=1,2$) be defined as in Table~\ref{tab:orth}.
\begin{enumerate}
\setcounter{enumi}{-1}
\item\label{thm:quotdetail:gen} $\Delta$ is generated by $\Omega$ and $X_0$.
\item\label{thm:quotdetail:pres} $P_1 = \langle X_1 \mid  \sR_1\rangle$ is a presentation for $G$.
The image of $g \in \Delta$ as a canonical word in $P_1$ can be computed in $O(C_1)$ field operations. 
\item\label{thm:quotdetail:pc} $P_2= \langle X_2 \mid  \sR_2\rangle$ is a polycyclic presentation for $G$.
The image of $g \in \Delta$ as a canonical word in $P_2$ can be computed in $O(C_1)$ field operations plus $C_2$ discrete logarithms.
\item\label{thm:quotdetail:coset}  A canonical representative of the coset $\Omega g$, where $g \in \Delta$,  can be computed in $O(C_3)$ field operations.
\end{enumerate}
For unitary and orthogonal groups, these algorithms are Las Vegas; in the other cases they are deterministic.
\end{thm}

\begin{sidewaystable} \caption{Presentations and complexity for classical groups}\label{tab:orth}
{\small

\begin{tabular}{l||l||l|l|l||l|l|l||l}
	&	& \multicolumn{3}{c||}{Presentation  $P_1$} 
					& \multicolumn{3}{c||}{Presentation $P_2$}\\
$\Omega$&$X_0$ $(1)$	& $X_1$ $(2)$ &$\sR_1$& $C_1$	& $X_2$	&$\sR_2$& $C_2$	& $C_3$ \\ \hline\hline
$\sl_d(q)$ 
	& $A(\la) = \diag(\la,1,\dots,1)$
		& $a(\la)$
			& $(3)$	&$d^\omega$
					& $a:=a(\xi)$ 
						&(4), $a^{q-1}$ 
							& 1&$d^\omega$\\\hline
$\spl_d(q)$&	&	&	&	&	&	&	&\\
$q$ even &$A(\la) = \la^{q/2}I_d$  
		& $a(\la)$
			&$(3)$	&$d^2+\log q$& $a:=a(\xi)$
						&(4), $a^{q-1}$
							& 1 &$d^2 +\log q$\\
$q$ odd
	&$A(\la) =  (\la I_m\oplus I_m)^X $  
		& $a(\la)$
			&$(3)$	&$d^\omega$& $a:=a(\xi)$
						&(4), $a^{q-1}$
							& 1 &$d^\omega$\\\hline
$\su_d(q)$
	& $A(\la) = \la I_d$
		& $a(\la),$
			&$(3)$, $b(\lambda)^{q+1}$, 
				&$d^\omega + \log q$
					&$a:=a(\zeta),$
						&(4), $[a, b]$, 
							&2&$d^\omega + d^2\log q$\\
	& $B(\la)=\left((\lambda^q)\ds I_{d-2}\ds(\lambda^{-1})\right)^X$
		&$b(\la)$&$a(\lambda)^{q-1}=b(\lambda)^{d}$
				& $\qquad+ \log^2 p$
					&$b:=b(\zeta)$
						&$b^{q+1}$,
							&& $\qquad + d\log^2 p$ \\
        &    &   & $[a(\la),b(\mu)]$&  & &$a^{q-1}= b^{d}$  & &\\  \hline
$\Omega_d(q)$,
	& $R_0, C(\la) = \la^{q/2} I_d$
		&$r_0,c(\la)$
			& $(3)$, $[r_0, c(\lambda)]$
				&$d^\omega+\log^2 q$
					&$r_0,c:=c(\xi)$
						&(4), $[r_0,c]$,
							&1&$d^\omega+\log^2 q$\\
$q$ even        &     &       &     &     &     &   $c^{q-1}$   &   & \\\hline
$\Omega_d^\circ(q)$, 
	&$R_0, R_1,$ 
		&$r_0,r_1,$
			& $(3)$, $[r_i, c(\la)]$, 
				&$d^\omega+\log q$	
					&$r_0,r_1,$
						&(4), $[r_0,c]$, 
							&1&$d^\omega +d\log q$\\
$d$ odd, 
	&$C(\la) = \left(\la^2 I_m \ds (\la) \ds I_m\right)^X$	
		&$c(\la)$& $c(-1) = r_0$
				&	&$c:=c(\xi)$	
						&$[r_1, c]$,&\\
$q$ odd & &      &    &      &      &  $c^{(q-1)/2}$ & \\\hline
$\Omega_d^+(q)$, 
	&$R_0,R_1,$
		&$r_0,r_1,$
			& $(3)$, 
				&$d^\omega+d\log q$
					&$r_0,r_1,$
						&(4), $r_0^c=r_1$,
							&1&$d^\omega +d\log q$\\
$d$ even, 
	&$C(\la) = \left(\la I_m \ds I_m\right)^X$
		&$c(\la)$& $r_i^{c(\la)} = r_{i + \iota(\la)}$
				&	&$c:=c(\xi)$
						&$r_1^c = r_0$,&\\
$q$ odd        &       &        &      &       &      &$c^{q-1}= r_0$&&\\\hline
$\Omega_d^-(q)$, 
	&$R_0,R_1,$
		&$r_0,r_1,$
			& $(3)$, $[r_i, c(\la)]$, 
				&$d^\omega+d\log q$	
					&$r_0,r_1,$
						&(4), $r_0^c=r_1$,
							&1&$d^\omega +d\log q$\\
$d$ even, 
	&$C(\la) = \left(\la^2 I_{m} \ds \la I_2 \ds I_m\right)^X$
		&$c_0, c(\la)$	&$c(-1) = r_0 r_1$
				&	&$c:=c(\sqrt{\xi\gamma^{-1}})c_0$
						&$r_1^c=r_0$,&\\
$q$ odd	&$C^-_0  = \left(\gamma I_{m} \ds
	 \left(\begin{smallmatrix}0&1\\\gamma&0\end{smallmatrix}\right)\ds I_m\right)^X$
		&	& $[c_0, c(\la)], c_0^2 = c(\gamma), $
				&	& 	&$c^{q-1}=r_0$ &\\
	&	&	& $r_i^{c_0} = r_{i+1}$
				&	&	&&
\end{tabular} \\
\vspace{5mm}

$(1)$ The generators $R_0,R_1\in X_0$ are defined as in Proposition~\ref{prop:reflections}.
For the group $\Omega_d^-(q)$, we define $\gamma$ as in Proposition~\ref{prop:canelts}.\\
$(2)$ We define $a(\la) \in X_2$ to be the coset $\Omega A(\la)$, and similarly for $b(\la)$, $r_0$, $r_1$, $c(\la)$, $c_0$, for $\la,\mu\in\F_{q^u}^\times$ and $i\in\F_2^+$. \\
$(3)$ The following relations are in $\sR_1$ whenever the relevant generators are defined:\\
$\quad a(\la)a(\mu)=a(\la\mu)$, $b(\la)b(\mu)=b(\la\mu)$, $c(\la)c(\mu)=c(\la\mu)$,
$r_0^2=r_1^2=(r_0r_1)^2=1$.\\
$(4)$ The following relations are in $\sR_2$ whenever the relevant generators are defined:
$r_0^2=r_1^2=(r_0r_1)^2=1$.\\

}
\end{sidewaystable}

Note that Theorem~\ref{thm:quot} is just a simplified version of this result.
The proof is straightforward in the linear and symplectic cases, and is similar to the unitary case. 
\begin{proof}[Proof of Theorem~\ref{thm:quotdetail}, unitary case]  
Proof of (\ref{thm:quotdetail:gen}): By \cite[Table~2.1.C]{KL90}, $[\Delta : \Omega] = q^2-1$.
The matrix $A(\lambda) \in \Delta$ for all $\lambda \in \BBF_{q^2}^\times$, as $A(\lambda)$ preserves the canonical unitary form up to scalars. The matrix $B(\lambda) \in \gu_d(q)$ for all $\lambda \in \BBF_{q^2}^\times$, as it preserves the canonical unitary form. The determinant of $B(\zeta)$ has order $q+1$, so $B:= \langle B(\lambda), \Omega \rangle/\Omega$ is cyclic of order $q+1$. The $\tau$ map shows that $\langle A(\lambda), B \rangle/B$ is cyclic of order $q-1$, so the result follows. 

Proof of (\ref{thm:quotdetail:pres}): First we check the presentation $P_1$. 
Since $A(\la)A(\mu)=A(\la\mu)$, we see that $a(\la)a(\mu)=a(\la\mu)$, and similarly 
$b(\la)b(\mu)=b(\la\mu)$. It follows from the proof of (0) that $b(\lambda)^{q+1} = 1$, and that some power of $a(\la)$ is a power of $b(\la)$. To show that $a(\la)^{q-1}=b(\la)^d$, note that $A(\la)^{q-1}B(\la)^{-d}$ has determinant $1$.

We map $g\in\Delta$ to $a(\tau(g))b(\mu^{-d}\det(g))\in P_1$, where $\mu$ is the canonical solution of $\mu^{q+1}=\tau(g)$. This is the correct image  since it factors through $\det$ and 
$\tau$ correctly.  Since $\tau(g)$ can be computed by a deterministic algorithm in $O(d^2)$ field operations by Lemma~\ref{lem:tau}, and $\mu$ can be computed by a Las Vegas algorithm in $O(\log q+ \log^2 p)$ field operations by Proposition~\ref{prop:norm}, the result follows. 

Proof of (\ref{thm:quotdetail:pc}): It is clear that $P_2$ presents the same group as $P_1$. To write $g \in \Delta$ as a word in $a$ and $b$, find the discrete logarithms of $\tau(g)$ and $\mu^{-d} \det(g)$. 

Proof of (\ref{thm:quotdetail:coset}): Use Theorem~\ref{thm:fix_form} to find $X$ such that 
$\su_d(q,\beta)=\su_d(q)^X$.
Take the coset representative of $g\in\Delta$ to be 
$(A(\tau(g))B(\mu^{-d}\det(g)))^X$.
\end{proof}

In the remainder of this section, we consider the orthogonal case.  
Since $\Omega$ is quasisimple by assumption,  $d \geq 3$.  
If $q$ is even, we also assume $d$ is even, since in even characteristic the odd degree orthogonal groups are isomorphic to symplectic groups.
For $\epsilon \in \{+, -, \circ\}$ we write $G=G^\epsilon(q):= \co^\epsilon_d(q)/\Omega^\epsilon_d(q)$.

Our first result proves Theorem~\ref{thm:quotdetail}(\ref{thm:quotdetail:gen}), and part of Theorem~\ref{thm:quotdetail}(\ref{thm:quotdetail:pres}) for the orthogonal case. 

\begin{prop} \label{prop:cosets} 
The group $\co^\epsilon_d(q)$ is generated by $\Omega^\epsilon_d(q)$ together with the generators $X_0$ in Table~\ref{tab:orth}. 
Furthermore,  $P_1=\langle X_1|\mathcal{R}_1 \rangle$ is a presentation for $G^\epsilon_d(q)$.
\end{prop}

\begin{proof}
It is easy to check that $C^{\epsilon}(\lambda) \in \co_d^\epsilon(q)$
and $C^-_0 \in \co_d^-(q)$. Note that
$\tau(C^{\epsilon}(\la))=\la^2$ when $q$ is odd and $\epsilon$ is $\circ$ or $-$;
whilst $\tau(C^{\epsilon}(\la))=\la$ in all other cases. 
One may check that $\tau(C_0^-)=\gamma$.

The kernel of $\tau$ on $\co^{\epsilon}_d(q)$ is $\orth^\epsilon_d(q)$, and 
its image is $\F_q^\times$ if $d$ is even,  and $\F_q^{\times2}$ otherwise \cite[\S2.1]{KL90}. 
For $d$ odd,  $\tau(C^\circ(\xi))=\xi^2$ generates $\F_q^{\times2}$.
If $\epsilon$ is $+$ or $q$ is even,
then $\tau(C^\epsilon(\xi))=\xi$ generates $\F_q^\times$.
Finally, if $\epsilon$ is $-$ and $q$ is odd, then
$\tau(C^-(\xi))=\xi^2$ and $\tau(C_0^-)=\gamma$ generate $\F_q^\times$, since 
$\gamma$ is nonsquare.
Since $\orth^\epsilon_d(q)$ is generated by $\Omega^\epsilon_d(q)$ and the reflections,  $\co^\epsilon_d(q)$ is generated by the given elements.

For $q$ even or $d$ odd, $G^\epsilon(q) = 
\langle r_0\rangle \times \langle c(\xi)\rangle  \cong \F_2^+\times \F_q^\times$.
For $q$ odd,  
$G^+(q)$ is an extension of $\langle r_0,r_1\rangle\cong (\F_2^+)^2$ by $\langle c(\xi)\rangle \cong \F_q^\times$, whilst 
$G^-(q)$ is an extension of  $\langle r_0,r_1\rangle\cong (\F_2^+)^2$ by $\langle c(\xi), c_1\rangle \cong \F_q^\times$.
Hence $G^\epsilon(q)$ has the same order as $\co^{\epsilon}_d(q)/\Omega^{\epsilon}_d(q)$  \cite[\S~2.1]{KL90}.
It therefore suffices to show that the relations hold. 

All relations involving only $r_0$ and $r_1$ hold because the quotient $\orth^{\epsilon}_d(q)/\Omega^{\epsilon}_d(q)$ is an elementary abelian $2$-group. 
For the relations involving $r_0$ or $r_1$ conjugated by $c(\la)$ or
$c_0$, note that $\refl_v^g=\refl_{vg}$ for $v\in V$ and $g\in\co_d^\epsilon(q)$.
For $q$ even, all reflections are in the same coset of $\Omega_d^\pm(q)$,
and so $r_0^{c(\la)}=r_0$.
For $q$ odd, $\iota(Q(vg))=\iota(Q(v))+\iota(\tau(g))$. 
For the relations involving products and powers of $c(\la)$ and $c_0$, one checks that $C^\epsilon(\la)C^\epsilon(\mu) =C^\epsilon(\la\mu)$ and so $c(\la)c(\mu) =c(\la\mu)$.
Now,  $C_{2m+1}^\circ(-1)=I_m\ds(-1)\ds I_m = \refl_x$, and since  $Q^\circ(x) = 1$ we deduce that $c(-1)=r_0$.
Finally, 
$C^-(\la)$ commutes with 
$C^-_0$; $(C_0^-)^2=C^-(\gamma)$; and
$C^-(-1)=I_m\ds -I_2\ds I_m = \refl_{x}\refl_{y}$, so $c(-1)=r_0r_1$.
\end{proof}

By setting $c=c(\xi)$, or $c=c(\sqrt{\xi\gamma^{-1}})c_0$ for $q$ odd and  $\epsilon=-$, 
we get presentations for the same groups with a bounded number of generators and relations. %
\begin{cor} \label{cor:cosets}
$P_2=\langle X_2|\mathcal{R}_2 \rangle$ is a presentation for $G^\epsilon_d(q)$.%
\end{cor}

We can now prove Theorem~\ref{thm:quotdetail} for the orthogonal groups. If $q$ is odd and $Q$ is of $-$ type, we  assume that the discrete log of $\gamma$ has been precomputed in 
(\ref{thm:quotdetail:pc}). 
We only give the case where $q$ is odd, $d$ is even, and $Q$ is of $-$ type, as the other orthogonal cases are similar. 

\begin{proof}[Proof of Theorem~\ref{thm:quotdetail}, orthogonal minus case]
Proof of (\ref{thm:quotdetail:gen}): This is immediate from Proposition~\ref{prop:cosets}.

Proof of (\ref{thm:quotdetail:pres}): It is immediate from Proposition~\ref{prop:cosets} that $P_1$ presents $G^\epsilon_d(q)$. For the homomorphism, we first find a canonical matrix  $X$ which tranforms the canonical form to $F$, in $O(d^\omega + d^2\log q)$ field operations.
We  compute $\tau(g)$ in $O(d^2)$ field operations. 
If $\tau(g)$ is a square, we take 
$\la=\sqrt{\tau(g)}$, $z = c(\la)$ and $C= C^-(\la)$. 
Otherwise we take $\la=\sqrt{\tau(g)\gamma^{-1}}$, $z = c_0c(\la)$, and $C= C_0^-C^-(\la)$. 
We then let $h= g^{X^{-1}}C^{-1}$,  find $a= \det(h)$ and $b= \spin(h)$ in $O(d^\omega + \log q)$ field operations. We map $g$ to $r_0^{b'} r_1^b z$, where $b' = b$ if $a = 1$ and $b' = b+1$ otherwise. %

Proof of (\ref{thm:quotdetail:pc}): It is immediate from Corollary~\ref{cor:cosets} that $P_2$ presents $G^\epsilon_d(q)$. For the homomorphism, find $k = \log_{\xi\gamma} \la=\frac{\log\la}{\log\gamma+1}$ with a discrete log call, and map $g$ to $r_0^{b'} r_1^b c^k$. 

Proof of (\ref{thm:quotdetail:coset}): Write down  $R_0$ and $R_1$ from
Proposition~\ref{prop:reflections} in $O(d^\omega + \log q)$ field operations, then the 
representative is $(R_0^{b'} R_1^b C)^X$.  
\end{proof}

We finish with a special case, where our algorithms run faster. 

\begin{prop}\label{prop:decompose_go}
Let $Q$ be a nondegenerate quadratic form, and let $g \in \orth_d(q, Q)$.
Then the image of $g$ under the natural homomorphism to $\BBF_2^+$ ($q$ even) or $(\BBF_2^+)^2$ ($q$ odd)  can be found by a deterministic algorithm in $O(d^\omega)$ field operations ($q$ even) or $O(d^\omega +  \log q)$ field operations ($q$ odd) . 
A canonical coset representative for $g$
can then be constructed by a deterministic algorithm in $O(d^2)$ field operations if $q$ is even and, given $\zeta$, by a  Las Vegas algorithm in $O(d^\omega + \log q)$ field operations otherwise. 
\end{prop}

\section{Applications: conjugacy and maximal subgroups}\label{sect:app}
Given a finite group $G$, the basic conjugacy problems are:
\begin{enumerate}
\item\label{prob:reps} find a set of canonical representatives of the conjugacy classes of $G$;
\item\label{prob:conj} given $x\in G$, find $g\in G$ such that $x^g$ is a canonical class representative;
and
\item\label{prob:cent} given a class representative $x$, find generators for $\CC_G(x)$.
\end{enumerate}
We conjugate to a class representative in problem~\ref{prob:conj}, rather than designing an algorithm to conjugate arbitrary pairs of elements,
because it reduces memory requirements.  
This way we need only work with a single element of the group, since the representative itself is implicit in the algorithm but does not usually
need to be written down.
This was our motivatation for the inclusion of canonical coset representatives in Theorem~\ref{thm:quotdetail}(\ref{thm:quotdetail:coset}). 

Suppose we can solve the element conjugacy problem in the group $\Delta$.
We  briefly describe how to solve the same problem for groups $G$ with $\Omega\le G\le \Delta$.
This is a slight generalisation of the results of \cite{Wall80}, and is based on the following lemma.
\begin{lem}
Let $\Delta$ be a group, $A$ a finite group, and $\phi:\Delta\to A$ an epimorphism.
Let $\Omega$ be the kernel of $\phi$.  
Suppose $G$ is a group with $\Omega\le G\normeq \Delta$.
Given $g \in G$, the $G$-classes contained in $g^\Delta$ correspond
to the elements of $A/\phi(\CC_\Delta(g)G)$ under the map
$$
  (g^h)^\Delta \mapsto \phi(\CC_\Delta(g)Gh)
$$
for $h$ in $\Delta$.
\end{lem}
\begin{proof}
Clearly every $G$-class in $g^\Delta$ is of the form $(g^h)^G$ for some $h\in \Delta$.
Now $(g^h)^G=(g^{h'})^G$ if and only if  $g^{hg'}=g^{h'}$ for some $g'\in G$, that is,  $hg'{h'}^{-1}$ is in $\CC_\Delta(g)$ for some $g'\in G$.
Since $G$ is normal in $\Delta$, this is equivalent to $h$ being in $\CC_\Delta(g)Gh'$, 
which means $\CC_\Delta(g)Gh=\CC_\Delta(g)Gh'$.
Since $A/\phi(\CC_\Delta(g)G)$ is naturally isomorphic to $\Delta/\!\CC_\Delta(g)G$, we are
done. 
\end{proof}
\noindent
Hence, in order to compute the classes in $G$ from the classes in $\Delta$, we need to know the images of centralisers under $\phi$ and we need representatives
$h_a\in\phi^{-1}(a)$ for all $a\in A$.
If $G$ is not normal in $\Delta$, we need to apply this lemma more than once:
since $\Delta/\Omega$ is soluble for classical groups $\Omega$, every $G$ with $\Omega\le G\le \Delta$ is subnormal in $\Delta$.

Solving problem (\ref{prob:reps}) is only possible for relatively small groups, but since Theorem~\ref{thm:quotdetail}(\ref{thm:quotdetail:coset}) gives canonical \emph{coset} representatives we can find canonical \emph{class} representatives to solve problem (\ref{prob:conj}) without first solving (\ref{prob:reps}).
Canonical class representatives also simplify the centraliser problem (\ref{prob:cent}), and allow us to compare results between different runs of the algorithms.
A detailed description of these algorithms is given in \cite{HallerMurray07Pre}.

An important application of Theorem~\ref{thm:fix_form} is to the construction of
maximal subgroups of classical groups, as in \cite{HRD, HRD10}.
When writing down generating matrices for a maximal subgroup, it is often convenient to construct initial matrices which preserve a form other than Magma's canonical classical form. We then conjugate the matrices so that they preserve the correct form. 
Since the isometry construction algorithm given in \cite{HRD} does not return the same conjugating matrix each time,  different conjugates of the maximal subgroup are found each time it is constructed. Using
Theorem~\ref{thm:fix_form}, the \emph{same} subgroup 
can now be constructed each time. This is not essential, but is often useful: for example when investigating containments between subgroups.

\section{Timings}\label{sect:timings}
In this section we present two tables of timings data for a Magma 2.14-9~\cite{Magma07} implementation of our algorithms. We tested our spinor norm algorithm on $\orth_d(q, Q)$ on all five cases: odd dimension and odd characteristic, and both types of form in even dimensions in both even and odd characteristic. In each case we computed the spinor norm of a random element of a random conjugate of the general orthogonal group. 
 
Next we tested the canonical coset representative algorithms on all five cases. We took a random conjugate of the conformal orthogonal group, and then selected a random element.   The time to find coset representatives for elements of the general orthogonal group lies between that taken to compute the spinor norm and to find coset representatives in the conformal orthogonal group. 
 
The experiments were carried out on a 1.5 GHz PowerPC G4 processor. The machine has 1.25GB of RAM, but memory was not a factor. All times are given in milliseconds, and are the average of 50 trials; the symbol -- indicates that the average time was less than $1$ millisecond.
 
As we would expect, the time required grows extremely slowly with $q$, and somewhat more quickly with $d$. Far less time is required for even $q$ than odd $q$.
Notice however that the representation of the field is more significant than its size, as $3^{16}$ is only about four times larger than $10000019$, yet the tests always take far longer.

\begin{table} \caption{Spinor norm on $\orth^{\epsilon}_d(q, Q)$} \label{ta:spinor} \vspace{-4mm}
$\begin{array}{ll | rrrrr | rrr |rrrrr}
&&&&p&&&3^i&&&2^i&&&&\\
\mbox{Type} & d & 5 & 17 & 47 & 73 & 10000019 & 3^6 & 3^{11}&  3^{16} &  2^5 &  2^{10} &  2^{20} & 2^{40} & 2^{80}  \\
\hline
\circ & 15 & 1 & 1 & 1 & 1 & 1 & 1 & 2 & 5 &&&  &  &  \\
& 55 & 4 & 9 & 9 & 9 & 11 & 11 & 28 & 184 &  &  &  &  &  \\ 
& 95 & 11 & 27 & 27 & 28 & 34 & 45 & 140 & 1083 &  &  &  &  &  \\
\hline
+ & 20 & 1 & 1 & 1 & 1 & 1 & 1 & 3 & 10 & - & - & - & 4 & 4 \\
& 60 & 4 & 11 & 10 &11 & 13 & 13 & 38 & 246 & - & 1 & 12 & 60 & 78 \\
& 100 &12 & 28 & 28 &  27 &  33 & 50 & 153 & 1408  & 2 & 7 & 57 & 311 & 413 \\
\hline
- & 20 & 1 & 1 & 2 & 1 & 1 &  1 & 3 &  10  &  - & - & - & 4 & 3 \\
& 60 & 4 & 11 & 11 & 11 & 14 & 14 &  36 & 256 &  -& 1 & 12 & 60 & 82  \\
& 100 & 11 & 28 & 27 & 26 & 33 & 48 & 148  & 1373  & 4 & 7 & 56 &  289 &  390\\
\end{array}$
\end{table}
\begin{table} \caption{Coset representatives in $\co^{\epsilon}_d(q, Q)$} \label{ta:coset_odd}  \vspace{-4mm} $\begin{array}{ll | rrrrr | rrr | rrrrr}
&&&&p&&&3^i \\
\mbox{Type} & d & 5 & 17 & 47 & 73 & 10000019 & 3^6 & 3^{11}&  3^{16} & 2^5 &  2^{10} &  2^{20} & 2^{40} & 2^{80} \\
\hline
\circ & 15 & 3 & 4 & 4 & 4 & 6 & 3 & 5 & 13 &&&&& \\
& 55 &  33 & 48 & 55 & 47 & 59 & 46 & 72 & 392 &&&&&\\
& 95 & 147 & 201 & 184 & 176 & 211 & 189 & 317 & 2342 &&&&&\\
\hline
+ & 20 &  6  & 7 & 7 & 7 & 10 & 7 & 10 & 34  & 1 & 2 & 4 & 8 & 14\\
& 60 & 46 & 62 & 68 & 65 & 77 & 76 & 148 & 936  & 17 & 18 & 26 & 124 & 170 \\
& 100 & 168 & 224 & 209 & 226 & 257 & 305 & 627 & 5645 & 49 & 67 & 127 & 553 & 629 \\
\hline
- & 20 & 7 & 9 &  9 & 9 & 11 & 153 & 15 & 40  & 1 & 1 &  3 & 7 & 11\\
& 60 & 50 & 72 & 71 & 70 & 90 & 244 & 196 & 1168 & 14 & 12 & 25 & 131 & 154\\
& 100 & 153 & 225 & 217 & 229 & 257 & 474 & 799 & 7969  & 71 & 60 & 119 & 553 & 736  \\
\end{array}$
\end{table}

\bibliographystyle{alpha}
\bibliography{spinor}

\begin{thebibliography}{JLPW95}

\bibitem[Asc84]{Asch}
M.~Aschbacher.
\newblock On the maximal subgroups of the finite classical groups.
\newblock {\em Invent. Math.}, 76(3):469--514, 1984.

\bibitem[Bab97]{Babai97}
L{\'a}szl{\'o} Babai.
\newblock Randomization in group algorithms: conceptual questions.
\newblock In {\em Groups and computation, II (New Brunswick, NJ, 1995)}, pages
  1--17. Amer. Math. Soc., Providence, RI, 1997.

\bibitem[BC07]{Magma07}
W.~Bosma and J.J. Cannon.
\newblock {\em Handbook of Magma functions}.
\newblock School of Mathematics and Statistics, University of Sydney, Sydney,
  2.14 edition, 2007.

\bibitem[BCS97]{Burgisser90}
P.~B{\"u}rgisser, M.~Clausen, and M.~A. Shokrollahi.
\newblock {\em Algebraic complexity theory}, volume 315 of {\em Grundlehren der
  Mathematischen Wissenschaften}.
\newblock Springer-Verlag, Berlin, 1997.

\bibitem[BGM93]{BlakeGaoMullin93}
I.F. Blake, S.~Gao, and R.C. Mullin.
\newblock Explicit factorization of {$x\sp {2\sp k}+1$} over {${\bf F}\sb p$}
  with prime {$p\equiv 3\bmod 4$}.
\newblock {\em Appl. Algebra Engrg. Comm. Comput.}, 4(2):89--94, 1993.

\bibitem[Bri06]{Britnell06B}
John~R. Britnell.
\newblock Cyclic, separable and semisimple transformations in the finite
  conformal groups.
\newblock {\em J. Group Theory}, 9(5):571--601, 2006.

\bibitem[Bro01]{BrooksOmega}
P.A. Brooksbank.
\newblock A constructive recognition algorithm for the matrix group
  {$\Omega(d,q)$}.
\newblock In {\em Groups and computation, III}, volume~8 of {\em Ohio State
  Univ. Math. Res. Inst. Publ.}, pages 79--93. de Gruyter, Berlin, 2001.

\bibitem[Bro03]{BrooksClassical}
P.A. Brooksbank.
\newblock Constructive recognition of classical groups in their natural
  representation.
\newblock {\em J. Symbolic Comput.}, 35(2):195--239, 2003.

\bibitem[CHM08]{CohenHallerMurray08}
Arjeh~M. Cohen, Sergei Haller, and Scott~H. Murray.
\newblock Computing in unipotent and reductive algebraic groups.
\newblock {\em LMS J. Comput. Math.}, 11:343--366, 2008.

\bibitem[GCL92]{Geddes92}
K.O. Geddes, S.R. Czapor, and G.~Labahn.
\newblock {\em Algorithms for computer algebra}.
\newblock Kluwer Academic Publishers, Boston, MA, 1992.

\bibitem[Gro02]{Grove02Book}
L.C. Grove.
\newblock {\em Classical groups and geometric algebra}, volume~39 of {\em
  Graduate Studies in Mathematics}.
\newblock American Mathematical Society, Providence, RI, 2002.

\bibitem[HM]{HallerMurray07Pre}
Sergei Haller and Scott~H. Murray.
\newblock Computing conjugacy in finite classical groups.
\newblock Unpublished.

\bibitem[HRD05]{HRD}
D.F. Holt and C.M. Roney-Dougal.
\newblock Constructing maximal subgroups of classical groups.
\newblock {\em LMS J. Comput. Math.}, 8:46--79, 2005.

\bibitem[HRD10]{HRD10}
D.F. Holt and C.M. Roney-Dougal.
\newblock Constructing maximal subgroups of orthogonal groups.
\newblock {\em LMS J. Comput. Math.}, 2010.
\newblock To appear.

\bibitem[JLPW95]{ModAtlas}
C.~Jansen, K.~Lux, R.~Parker, and R.~Wilson.
\newblock {\em An Atlas of Brauer Characters}.
\newblock Oxford University Press, Oxford, UK, 1995.

\bibitem[KL90]{KL90}
P.~Kleidman and M.~Liebeck.
\newblock {\em The subgroup structure of the finite classical groups}.
\newblock Cambridge University Press, Cambridge, 1990.

\bibitem[Lan93]{Lang93Book}
Serge Lang.
\newblock {\em Algebra}.
\newblock Addison-Wesley Publishing Co., Reading, Mass., third edition, 1993.

\bibitem[LG01]{Leedham-Green01}
C.R. Leedham-Green.
\newblock The computational matrix group project.
\newblock In {\em Groups and computation, III}, volume~8 of {\em Ohio State
  Univ. Math. Res. Inst. Publ.}, pages 229--247. de Gruyter, Berlin, 2001.

\bibitem[L{\"u}b]{LuebeckWebpage}
F.~L{\"u}beck.
\newblock \verb1http://www.math.rwth-aachen.de/~Frank.Luebeck/data/ConwayPol1.

\bibitem[Str69]{Strassen}
V.~Strassen.
\newblock Gaussian elimination is not optimal.
\newblock {\em Numer. Math.}, 13:354--356, 1969.

\bibitem[Tay92]{Taylor}
D.E. Taylor.
\newblock {\em The geometry of the classical groups}.
\newblock Heldermann Verlag, Berlin, 1992.

\bibitem[Wal80]{Wall80}
G.~E. Wall.
\newblock Conjugacy classes in projective and special linear groups.
\newblock {\em Bull. Austral. Math. Soc.}, 22(3):339--364, 1980.

\end{thebibliography}

\end{document}